\documentclass{article}%
\usepackage{amsmath}
\usepackage{amsfonts}
\usepackage{amssymb}
\usepackage{graphicx}%
\providecommand{\U}[1]{\protect\rule{.1in}{.1in}}

\newtheorem{theorem}{Theorem}

\newtheorem{corollary}[theorem]{Corollary}

\newtheorem{lemma}[theorem]{Lemma}

\newtheorem{proposition}[theorem]{Proposition}

\newenvironment{proof}[1][Proof]{\noindent\textbf{#1.} }{\ \rule{0.5em}{0.5em}}
\begin{document}

\title{Global minimizers of the Allen-Cahn equation in dimension $n\geq8$}
\author{Yong Liu\\School of Mathematics and Physics, \\North China Electric Power University, Beijing, China,\\Email: liuyong@ncepu.edu.cn
\and Kelei Wang\\School of Mathematics and Statistics, Wuhan Univeristy\\Email:wangkelei@whu.edu.cn
\and Juncheng Wei\\Department of Mathematics, \\University of British Columbia, Vancouver, B.C., Canada, V6T 1Z2\\Email: jcwei@math.ubc.ca}
\maketitle

\section{Introduction and main results}

The bounded entire solutions of the Allen-Cahn equation
\begin{equation}
-\Delta u=u-u^{3}\text{ \ in }\mathbb{R}^{n},\text{ }\left\vert u\right\vert
<1, \label{AC}%
\end{equation}
has attracted a lot of attentions in recent years, partly due to its intricate
connection to the minimal surface theory. For $n=1,$ $\left(  \ref{AC}\right)
$ has a heteroclinic solution $H\left(  x\right)  =\tanh\left(  \frac{x}%
{\sqrt{2}}\right)  .$ Up to a translation, this is the unique monotone
increasing solution in $\mathbb{R}.$ De Giorgi (\cite{DeG}) conjectured that
for $n\leq8,$ if a solution to $\left(  \ref{AC}\right)  $ is monotone in one
direction, then up to translation and rotation it must be one dimensional and
hence equals $H$ in certain coordinate. De Giorgi's conjecture is parallel to
the Bernstein conjecture in minimal surface theory, which states that if
$F:\mathbb{R}^{n}\rightarrow\mathbb{R}$ is a solution to the minimal surface
equation
\[
\operatorname{div}\frac{\nabla F}{\sqrt{1+\left\vert \nabla F\right\vert ^{2}%
}}=0,
\]
then $F$ must be a linear function in its variables. The Bernstein conjecture
has been proved to be true for $n\leq7$. The famous Bombieri-De Giorgi-Giusti
minimal graph (\cite{BDG}) gives a counter-example for $n=8,$ which also
disproves the Bernstein conjecture for all $n\geq8.$ As for the De Giorgi
conjecture, it has been proved to be true for $n=2$ (Ghoussoub-Gui \cite{G}),
$n=3$ (Ambrosio-Cabre \cite{C}), and for $4\leq n\leq8$ (Savin \cite{S}),
under an additional limiting condition
\[
\lim_{x_{n}\rightarrow\pm\infty}u\left(  x^{\prime},x_{n}\right)  =\pm1.
\]
This condition together with the monotone property ensures that $u$ is a
global minimizer in the sense that, for any bounded domain $\Omega
\subset{\mathbb{R}}^{n}$,
\[
J\left(  u\right)  \leq J\left(  u+\phi\right)  ,\text{ for all }\phi\in
C_{0}^{\infty}\left(  \Omega\right) ,
\]
where
\[
J\left(  u\right)  :=\int_{\Omega}\left[  \frac{1}{2}\left\vert \nabla
u\right\vert ^{2}+\frac{\left(  u^{2}-1\right)  ^{2}}{4}\right] ,
\]
see Alberti-Ambrosio-Cabre \cite{AC} and Savin \cite{S}. We also refer to
Farina-Valdinoci\cite{Farina} for discussion on related results.

On the other hand, it turns out that for $n\geq9,$ there indeed exist monotone
solutions which are not one dimensional. These nontrivial examples have been
constructed in $\cite{M2}$ using the machinery of infinite dimensional
Lyapunov-Schmidt reduction. The nodal set of these solutions are actually
close to the Bombieri-De Giorgi-Giusti minimal graph. Indeed, it is also
proved in \cite{M3} that for any nondegenerate minimal surfaces with finite
total curvature in ${\mathbb{R}}^{3}$, one could construct family of solutions
for the Allen-Cahn equation which \textquotedblleft follow\textquotedblright%
\ these minimal surfaces. These results provide us with further indication
that there is a deep relation between the minimal surface theory and the
Allen-Chan equation.

Regarding solutions which are not necessary monotone, in $\cite{S}$, Savin
also proved that if $u$ is a global minimizer and $n\leq7,$ then $u$ is one
dimensional. While the monotone solutions of Del Pino-Kowalczyk-Wei provides
examples of nontrivial global minimizers in dimension $n\geq9,$ it is not
known whether there are nontrivial global minimizers for $n=8.$ Due to the
connection with minimal surface theory, these global minimizers have long been
conjectured to exist in dimension 8 and higher. The existence of these global
minimizers will be our main focus in this paper.

To state our results, let us recall some basic facts from the minimal surface
theory. It is known that in $\mathbb{R}^{8},$ there is a minimal cone with one
singularity at the origin which minimizes the area, called Simons cone. It is
given explicitly by:%
\[
\left\{  x_{1}^{2}+...+x_{4}^{2}=x_{5}^{2}+...+x_{8}^{2}\right\}  .
\]
The minimality of this cone is proved in \cite{BDG}. More generally, if we
consider the so-called Lawson's cone ($2\leq i\leq j$)
\[
C_{i,j}:=\left\{  \left(  x,y\right)  \in\mathbb{R}^{i}\oplus\mathbb{R}%
^{j}:\left\vert x\right\vert ^{2}=\frac{i-1}{j-1}\left\vert y\right\vert
^{2}\right\}  ,
\]
then it has mean curvature zero except at the origin and hence is a minimal
hypersurface with one singularity. For $i+j\leq7,$ the cone is unstable
(Simons \cite{Simons}). Indeed, it is now known that for $i+j\geq8,$ and
$\left(  i,j\right)  \neq\left(  2,6\right)  ,$ $C_{i,j}$ are area minimizing,
and $C_{2,6}$ is not area minimizing but it is one sided minimizer. (See
\cite{Alen}, \cite{De}, \cite{Lin},\ \cite{Lawson}...). Note that the cone
$C_{i,j}$ has the $O\left(  i\right)  \times O\left(  j\right)  $ symmetry,
that is, it is invariant under the natural group actions of $O\left(
i\right)  $ on the first $i$ variables and $O\left(  j\right)  $ on the last
$j$ variables. We also refer to \cite{L} and references therein for more
complete history and details on related subjects.

It turns out there are analogous objects as the cone $C_{i,i}$ in the theory
of Allen-Cahn equation. They are the so-called saddle-shaped solutions, which
are solutions in $\mathbb{R}^{2i}$ of $\left(  \ref{AC}\right)  $ vanishes
exactly on the cone $C_{i,i}$ (Cabre-Terra \cite{C1,C2} and Cabre \cite{C3}).
We denote them by $D_{i,i}.$ It has been proved in \cite{C1} that these
solutions are unique in the class of symmetric functions. Furthermore in
$\cite{C1,C2}$ it is proved that for $2\leq i\leq3,$ the saddle-shaped
solution is unstable, while for $i\geq7,$ they are stable(\cite{C3}). It is
conjectured that for $i\geq4,$ $D_{i,i}$ should be a global minimizer. This
turns out to be a difficult problem. We show however in this paper the following

\begin{theorem}
\label{Main}Suppose either (1) $i+j\geq9$ or (2) $i+j=8$, $\left\vert
i-j\right\vert \leq4.$ Then there is a family of global minimizers of the
Allen-Cahn equation in $\mathbb{R}^{i+j}$ having $O\left(  i\right)  \times
O\left(  j\right)  $ symmetry and are not one dimensional. The zero level set
of these solutions converge to the cone $C_{i,j}$ at infinity.
\end{theorem}

More detailed asymptotic behavior of this family of solutions could be found
in Proposition \ref{asym} below. This family of solutions could be
parametrized by the closeness of its zero-level set to the minimizing cones.
Recall that a result of Jerison and Monneau (\cite{JM}) proved that the
existence of a nontrivial global minimizer in $\mathbb{R}^{8}$ which is even
in all of its variables implies the existence of a family of counter-examples
for the De Giorgi conjecture in $\mathbb{R}^{9}.$ Hence an immediate corollary
of Theorem \ref{Main} is the following

\begin{corollary}
Suppose either $i+j\geq9$ or $i+j=8$ with $\left\vert i-j\right\vert \leq4.$
There is a family of monotone solutions to the Allen-Cahn equation $\left(
\ref{AC}\right)  $ in $\mathbb{R}^{i+j+1},$ which is not one-dimensional and
having $O\left(  i\right)  \times O\left(  j\right)  $ symmetry in the first
$i+j$ variables.
\end{corollary}

This corollary could be regarded as a parallel result due to Simon
\cite{Simon} on the existence of entire minimal graphs. Our idea of the proof
is quite straightforward. We shall firstly construct minimizers on bounded
domains, with suitable boundary conditions. As we enlarge the domain, we will
see that a subsequence of solutions on these bounded domains will converge to
a global minimizer, as one expected. To ensure that the solutions converge, we
will use the family of solutions constructed by Pacard-Wei\cite{PW} as
barriers. The condition that the cone we start with is strict area minimizing
is used to ensure that the solutions of Pacard-Wei are ordered. To show the
compactness and precise asymptotic behavior we use the convenient tool of
Fermi coordinate. The rest of the paper is devoted to the proof of Theorem
\ref{Main}.

\medskip

\textit{Acknowledgement.} The research of J. Wei is partially supported by
NSERC of Canada. Part of the paper was finished while Y. Liu was visiting the
University of British Columbia in 2016. He appreciates the institution for its
hospitality and financial support. K. Wang is supported by ``the Fundamental
Research Funds for the Central Universities". Y. Liu is partially supported by the Fundamental
Research Funds for the Central Universities 2014zzd10 and 13MS39".

\section{Solutions on bounded domains and their asymptotic behavior}

Let us first of all deal with the case that the cone is the Simons cone in
$\mathbb{R}^{8}.$ The starting point of our construction of global minimizers
will be the solutions of Pacard-Wei $\left(  \cite{PW}\right)  $ which we
describe below.

Let $\nu\left(  \cdot\right)  $ be a choice of the unit normal of the Simons
cone $C_{4,4}.$ Since we are interested in solutions with $O\left(  4\right)
\times O\left(  4\right)  $ symmetry, let us introduce
\begin{equation}
r=\sqrt{x_{1}^{2}+...+x_{4}^{2}}, \quad s=\sqrt{x_{5}^{2}+...+x_{8}^{2}},\quad
l=\sqrt{x_{1}^{2}+...+x_{8}^{2}}. \label{2.1n}%
\end{equation}
There is a smooth minimal surface $\Gamma^{+}$ lying in one side of the Simons
cone which is asymptotic to this cone and has the following properties (see
\cite{HS}). $\Gamma^{+}$ is invariant under the group of action of $O\left(
4\right)  \times O\left(  4\right)  .$ Outside of a ball, $\Gamma^{+}$ is a
graph over $C_{4,4}$ and
\[
\Gamma^{+}=\left\{  X+\left[  r^{-2}+O\left(  r^{-3}\right)  \right]
\nu\left(  X\right)  :X\in C_{4,4}\right\}  ,\text{ as }r\rightarrow+\infty.
\]
Similarly, there is a smooth minimal hypersurface $\Gamma^{-}$ in the other
side of the cone. For $\lambda\mathbb{\geq}0,$ let $\Gamma_{\lambda}^{\pm
}=\lambda\Gamma^{\pm}$ be the family of homotheties of $\Gamma^{\pm}.$ Then it
is known that $\Gamma_{\lambda}^{\pm}$ forms a foliation of $\mathbb{R}^{8}.$
We use $s=f_{\lambda}\left(  r\right)  $ to denote these minimal surfaces.

For $\lambda$ sufficiently large, say $\lambda\geq\lambda_{0},$ by a
construction of Pacard-Wei $\left(  \cite{PW}\right)  $, there exist solutions
$U_{\lambda}^{\pm}$ whose zero level set is close to $\Gamma_{\lambda}^{\pm}.$
Moreover, they depend continuously on the parameter $\lambda$ and are ordered.
That is,
\begin{align*}
U_{\lambda_{1}}^{+}\left(  X\right)   &  <U_{\lambda_{2}}^{+}\left(  X\right)
,\lambda_{1}<\lambda_{2}.\\
U_{\lambda_{1}}^{-}\left(  X\right)   &  <U_{\lambda_{2}}^{-}\left(  X\right)
,\lambda_{1}>\lambda_{2},\\
U_{\lambda_{0}}^{-}\left(  X\right)   &  <U_{\lambda_{0}}^{+}\left(  X\right)
.
\end{align*}
We use $N_{u}$ to denote the zero level set of a function $u.$ Suppose that in
the $r$-$s$ plane we have
\[
N_{U_{\lambda}^{\pm}}=\left\{  \left(  r,s\right)  :s=F_{\lambda}^{\pm}\left(
r\right)  \right\}  .
\]
Then we have the following asymptotic behavior$:$%
\[
F_{\lambda}^{\pm}\left(  r\right)  =r\pm\frac{\lambda^{3}}{r^{2}}+o\left(
r^{-2}\right)  ,\text{ as }r\rightarrow+\infty.
\]
It should be emphasized that the construction in \cite{PW} only gives us these
solutions when $\lambda$ is sufficiently large.

\begin{proposition}
\label{prop 3} As $\lambda\to+\infty$, $U^{\pm}_{\lambda}\to\pm1$ uniformly on
any compact set of ${\mathbb{R}}^{8}$.
\end{proposition}

\begin{proof}
Denote $\varepsilon=\lambda^{-1}$ and $u_{\varepsilon}^{\pm}(X):=U_{\lambda
}^{\pm}(\lambda X)$. Then $u_{\varepsilon}^{\pm}$ are solutions to the
singularly perturbed Allen-Cahn equation
\[
-\varepsilon\Delta u_{\varepsilon}=\frac{1}{\varepsilon}\left(  u_{\varepsilon
}-u_{\varepsilon}^{3}\right)  .
\]
Moreover, the construction of \cite{PW} implies that $\{u_{\varepsilon}^{\pm
}=0\}$ lies in an $O(\varepsilon)$ neighborhood of $\Gamma^{\pm}$. Because the
distance from the origin to $\Gamma^{\pm}$ is positive, by the equation, we
see $u_{\varepsilon}^{\pm}$ is close to $\pm1$ in a fixed ball around the
origin. Rescaling back we finish the proof.
\end{proof}

\subsection{Minimizing arguments and solutions with $O\left(  4\right)  \times
O\left(  4\right)  $ symmetry}

For each $a\in\mathbb{R},$ we would like to construct a solution whose zero
level set in the $r$-$s$ plane is asymptotic to the curve
\[
s=r+ar^{-2}%
\]
at infinity. Without loss of generality, let us assume $a\geq0.$

Consider the first quadrant of the $r$-$s$ plane. Let $\left(  l,t\right)  $
be the Fermi coordinate around the minimal surface $\Gamma_{a}^{+},$ where $t$
is the signed distance to $\Gamma_{a}^{+}.$ Keep in mind that this Fermi
coordinate is not smoothly defined on the whole space $\mathbb{R}^{8}$. (It
may not be smooth around the axes.) For each $d$ large, let $L_{d}$ be the
line orthogonal to $\Gamma_{a}^{+}$ at the point $\left(  d,f_{a}\left(
d\right)  \right)  .$ Denote the domain enclosed by $L_{d}$ and the $r,s$ axes
by $\Omega_{d}$. (This domain should be considered as a domain in
${\mathbb{R}}^{8}$. We still denote it by $\Omega_{d}$ for notational
simplicity). We shall impose Neumann boundary condition on $r,s$ axes and
suitable Dirichlet boundary condition on $L_{d},$ to get a minimizer for the
energy functional.

Let $H_{d}^{\ast}\left(  \cdot\right)  $ be a smooth function defined on
$L_{d},$ equal to $H\left(  t\right)  $ away from the $r$-$s$ axes, where $H$
is the one dimensional heteroclinic solution. Since $\int_{\mathbb{R}}t
H^{^{\prime}}\left(  t\right)   ^{2}dt=0$, there exists a unique solution of
the problem
\[
\left\{
\begin{array}
[c]{l}%
-\eta^{\prime\prime}+\left(  3H^{2}-1\right)  \eta=-tH^{\prime},\\
\int_{\mathbb{R}}\eta H^{\prime}=0.
\end{array}
\right.
\]
This solution will be denoted by $\eta(\cdot)$. (There is an explicit form for
$\eta$, see \cite{M2}, but we will not use this fact.)

Let $\varepsilon>0$ be a small constant. Let $\rho$ be a cut-off function
defined outside the unit ball, equal to $1$ in the region $\left\{
\varepsilon s<r<\varepsilon^{-1}s\right\}  ,$ equal to $0$ near the $r,s$
axes. It is worth pointing out that the Fermi coordinate is smoothly defined
in the region $\left\{  \varepsilon s<r<\varepsilon^{-1}s\right\}  \backslash
B_{R}\left(  0\right)  ,$ for $R$ sufficiently large. We seek a minimizer of
the function $J$ within the class of functions
\[
S_{d}:=\left\{  \phi\in H^{1,2}\left(  \Omega_{d}\right)  :\phi=H_{d}^{\ast
}+\rho\eta\left(  t\right)  \left\vert A\right\vert ^{2}\text{ on }%
L_{d}\right\}  .
\]
Here $\left\vert A\right\vert ^{2}$ is the squared norm of the second
fundamental form of the minimal surface $\Gamma_{a}^{+}$ and hence it decays
like $O\left(  r^{-2}\right)  $ as $r$ tends to infinity. Slightly modifying
the function $H_{d}^{\ast}$ near the axes if necessary, using the asymptotic
expansion of the solutions $U_{\lambda}^{\pm},$ we could assume that%
\begin{equation}
U_{\lambda_{0}}^{-}<H_{d}^{\ast}+\rho\eta\left(  t\right)  \left\vert
A\right\vert ^{2}<U_{\lambda_{0}}^{+}\text{ on }\ L_{d}. \label{2.2n}%
\end{equation}

Let $u=u_{d}$ be a minimizer of the functional $J$ over $S_{d}.$ The existence
of $u$ follows immediately from standard arguments. But in principle, we may
not have uniqueness. Intuitively speaking, the uniqueness of minimizer should
be an issue related to minimizing property of the saddle-shaped solutions.

\begin{proposition}
$u_{d}$ is invariant under the natural action of $O\left(  4\right)  \times
O\left(  4\right)  .$
\end{proposition}

\begin{proof}
Let $e\in O\left(  4\right)  \times O\left(  4\right)  $. Then due to the
invariance of the energy functional and the boundary condition, $u\left(
e\cdot\right)  $ is still a minimizer. By elliptic regularity, $u$ is smooth.

Suppose $e$ is given by: $\left(  x_{1},x_{2},...,x_{8}\right)  \rightarrow
\left(  -x_{1},x_{2},...,x_{8}\right)  .$ We first show that $u\left(
x\right)  =u\left(  ex\right)  $ for any $x\in\Omega_{d}.$ Assume to the
contrary that this is not true. Let us consider the functions
\begin{align*}
w_{1}\left(  x\right)   &  :=\min\left\{  u\left(  x\right)  ,u\left(
ex\right)  \right\}  ,\\
w_{2}\left(  x\right)   &  :=\max\left\{  u\left(  x\right)  ,u\left(
ex\right)  \right\}  .
\end{align*}
Since $u\left(  \cdot\right)  $ and $u\left(  e\cdot\right)  $ have the same
boundary condition, $w_{1}$ and $w_{2}$ are also minimizers of the functional
$J.$ Hence they are solutions of the Allen-Cahn equation. Since $w_{1}\leq
w_{2}$ and $w_{1}\left(  0,x_{2},...,x_{8}\right)  =w_{2}\left(
0,x_{2},...,x_{8}\right)  ,$ by the strong maximum principle, $w_{1}=w_{2}.$
It follows that $u\left(  x\right)  =u\left(  ex\right)  .$

Let us use $y_{1}$ to denote the first four coordinates $\left(
x_{1},...,x_{4}\right)  $ and $y_{2}$ denote the last four coordinates
$\left(  x_{5},...,x_{8}\right)  .$ Suppose $\bar{e}$ is a reflection across a
three dimensional hyperplane $L$ in $\mathbb{R}^{4}$ which passes through
origin. This gives us a corresponding element in $O\left(  4\right)  \times
O\left(  4\right)  ,$ still denoted by $\bar{e},$%
\[
\bar{e}\left(  y_{1},y_{2}\right)  :=\left(  \bar{e}\left(  y_{1}\right)
,y_{2}\right)  .
\]
Similar arguments as above tell us that $u\left(  x\right)  =u\left(  \bar
{e}x\right)  $ for any $x\in\Omega_{d}.$ Because $O\left(  4\right)  \times
O\left(  4\right) $ is generated by reflections, this implies that $u$ is
invariant under the group of action of $O\left(  4\right)  \times O\left(
4\right)  .$
\end{proof}

Let us now fix a number $\lambda^{\ast}>\max\left\{  |a|^{\frac{1}{3}},\lambda_{0}\right\}  .$

\begin{lemma}
\label{barrier}For any $d$ large, there holds
\begin{equation}
\label{2.3nn}U_{\lambda^{\ast}}^{-}<u_{d}<U_{\lambda^{\ast}}^{+}.
\end{equation}

\end{lemma}

\begin{proof}
By Proposition \ref{prop 3}, for $\lambda$ sufficiently large, $u_{d}%
<U_{\lambda}^{+}$ on $\Omega_{d}.$ Now let us continuously decrease the value
of $\lambda.$ Since we have a continuous family of solutions $U_{\lambda}%
^{+},$ and for $\lambda>\lambda^{\ast},$ each of them is greater than $u_{d}$
on the boundary of $\Omega_{d}$ (recalling (\ref{2.2n}) and the ordering
properties of $U_{\lambda}^{\pm}$), hence by the strong maximum principle,
$u_{d}<U_{\lambda^{\ast}}^{+}.$ Similarly, one could show that $U_{\lambda
^{\ast}}^{-}<u_{d}.$ This finishes the proof.
\end{proof}

\subsection{Asymptotic analysis of the solutions\label{2.3}}

The minimizers $u_{d}$ are invariant under $O\left(  4\right)  \times O\left(
4\right)  $ action. By $\left(  \ref{2.3nn}\right)  ,$ we know that as $d$
goes to infinity, $u_{d}\rightarrow U$ in $C_{loc}^{2}$ for a nontrivial
entire solution $U$ to the Allen-Cahn equation.

We then claim

\begin{lemma}
$U$ is a global minimizer.
\end{lemma}

\begin{proof}
Let $\phi\in C_{0}^{\infty}({\mathbb{R}}^{n})$ be any fixed function. Then for
$d>2R_{0},$ where $supp(\phi)\subset B_{R_{0}}(0)$, we have
\begin{equation}
J(u_{d})\leq J(u_{d}+\phi)
\end{equation}
Letting $d\rightarrow+\infty$ we arrive at the conclusion.
\end{proof}

By Lemma \ref{barrier}, the nodal set $N_{U}$ of $U$ must lie between
$N_{U_{\lambda^{\ast}}^{+}}$ and $N_{U_{\lambda^{\ast}}^{-}}.$ We use
$s=F\left(  r\right)  $ to denote the nodal set of $U$.

The main result of this section is the following

\begin{proposition}
\label{asym}The zero level set of $U$ has the following asymptotic behavior:
\begin{equation}
F\left(  r\right)  -f_{a}\left(  r\right)  =o\left(  r^{-2}\right)  ,\text{ as
}r\rightarrow+\infty. \label{Asy}%
\end{equation}

\end{proposition}

The proof of this proposition relies on detailed analysis of the asymptotic
behavior of the sequence of solutions $u_{d}.$\medskip

To begin with, let us define an approximate solution as%
\[
\bar{H}\left(  l,t\right)  =\rho H\left(  t-h_{d}\left(  l\right)  \right)
+\left(  1-\rho\right)  \frac{H\left(  t-h_{d}\left(  l\right)  \right)
}{\left\vert H\left(  t-h_{d}\left(  l\right)  \right)  \right\vert },
\]
where $\rho$ is the cutoff function introduced before. We shall write the
solution $u_{d}$ in the Fermi coordinate as
\begin{equation}
u_{d}\left(  l,t\right)  =\bar{H}+\phi_{d}, \label{decom}%
\end{equation}
for some small function $h_{d}.$ Introduce the notation $\bar{H}^{\prime
}:=H^{\prime}\left(  t-h\right)  .$ We require the following orthogonality
condition on $\phi_{d}:$%
\[
\int_{\mathbb{R}}\phi_{d}\rho\bar{H}^{\prime}dt=0.
\]
Since $u_{d}$ is close to $H\left(  t\right)  ,$ we could find a unique small
function $h_{d}$ satisfying $\left(  \ref{decom}\right)  $ using implicit
function theorem for each fixed $l.$

Our starting point for the asymptotic analysis is the following estimate:

\begin{lemma}
\label{De1}The function $h_{d}$ and $\phi_{d}$ satisfy
\[
\left\vert h_{d}\right\vert +\left\vert \phi_{d}\right\vert +\left\vert
h_{d}^{\prime}\right\vert +\left\vert h_{d}^{\prime\prime}\right\vert \leq
Cl^{-2},
\]
where $C$ does not depend on $d.$
\end{lemma}

\begin{proof}
We first prove $\left\vert h_{d}\right\vert \leq Cl^{-2}.$ By the orthogonal
condition,
\begin{equation}
\int_{\mathbb{R}}\left(  u_{d}-\bar{H}\right)  \rho\bar{H}^{\prime}\left(
t\right)  dt=0. \label{or}%
\end{equation}
Hence%
\[
\int_{\mathbb{R}}\left[  u_{d}-H\left(  t\right)  \right]  \rho\bar{H}%
^{\prime}dt=\int_{\mathbb{R}}\left[  \bar{H}-H\left(  t\right)  \right]
\rho\bar{H}^{\prime}dt=h_{d}\int_{\mathbb{R}}\bar{H}^{\prime2}dt+o\left(
h_{d}\right)  .
\]
Since
\[
U_{\lambda^{\ast}}^{-}\leq u_{d}\leq U_{\lambda^{\ast}}^{+},
\]
and
\[
\left\vert U_{\lambda^{\ast}}^{\pm}-H\left(  t\right)  \right\vert \leq
Cl^{-2},
\]
then it holds that
\[
\left\vert h_{d}\right\vert \leq C\frac{\left\vert \int_{\mathbb{R}}\left[
u_{d}-H\left(  t\right)  \right]  \rho\bar{H}^{\prime}dt\right\vert }%
{\int_{\mathbb{R}}\bar{H}^{\prime2}dt}\leq Cl^{-2}.
\]
This also implies that
\[
\left\vert \phi_{d}\right\vert =\left\vert u_{d}-\bar{H}\right\vert \leq
Cl^{-2}.
\]

Next we show $\left\vert h_{d}^{\prime}\right\vert \leq Cl^{-2}.$ To see this,
we differentiate equation $\left(  \ref{or}\right)  $ with respect to $l.$
This yields
\[
\int_{\mathbb{R}}\left(  \partial_{l}u_{d}-\partial_{l}\bar{H}\right)
\rho\bar{H}^{\prime}\left(  t\right)  dt=-\int_{\mathbb{R}}\left(  u_{d}%
-\bar{H}\right)  \left[  \partial_{l}\rho\bar{H}^{\prime}\left(  t\right)
+\rho\partial_{l}\bar{H}^{\prime}\left(  t\right)  \right]  dt.
\]
As a consequence,
\begin{equation}
\left\vert h_{d}^{\prime}\int_{\mathbb{R}}\rho\bar{H}^{\prime2}\left(
t\right)  dt\right\vert \leq C\left\vert \int\partial_{l}u_{d}\rho\bar
{H}^{\prime}dt\right\vert +O\left(  l^{-2}\right)  . \label{hprime}%
\end{equation}
Observe that $u_{d}$ is a solution trapped between $U_{\lambda^{\ast}}^{+}$
and $U_{\lambda^{\ast}}^{-}.$ Hence elliptic regularity tells us $\partial
_{l}u_{d}=O\left(  l^{-2}\right)  .$ This together with $\left(
\ref{hprime}\right)  $ yield $\left\vert h_{d}^{\prime}\right\vert \leq
Cl^{-2}.$ Similarly, we can prove that $\left\vert h_{d}^{\prime\prime
}\right\vert \leq Cl^{-2}.$
\end{proof}

The Laplacian operator $\Delta$ has the following expansion in the Fermi
coordinate $\left(  l,t\right)  $:%
\[
\Delta=\Delta_{\Gamma^{t}}+\partial_{t}^{2}-M_{t}\partial_{t}.
\]
Here $M_{t}$ is the mean curvature of the surface
\[
\Gamma^{t}:=\left\{  X+t\nu\left(  X\right)  ,X\in\Gamma_{a}^{+}\right\}  .
\]
We use $g_{i,j}^{t}$ to denote the induced metric on the surface $\Gamma^{t}.$
Then
\begin{equation}
\Delta_{\Gamma^{t}}\varphi=\frac{\partial_{i}\left(  g^{i,j,t}\sqrt{\left\vert
g^{t}\right\vert }\partial_{j}\varphi\right)  }{\sqrt{\left\vert
g^{t}\right\vert }}. \label{Laplacian}%
\end{equation}
Here we have used $\left\vert g^{t}\right\vert $ to denote the determinant of
the metric tensor. For $t=0,$ $g_{i,j}^{0}:=g_{ij}$ is close to the metric on
the Simons cone, which has the form
\[
dl^{2}+l^{2}ds^{2},
\]
where $ds^{2}$ is the metric on $S^{3}\times S^{3}.$ In general, when
$t\neq0,$ the metric on $\Gamma^{t}$ and $\Gamma^{0}$ is in $\left(
l,t\right)  $ coordinate is related by%
\[
g^{t}=g^{0}\left(  I-tA\right)  ^{2}.
\]
In particular,
\[
g_{i,j}^{t}=g_{i,j}+O\left(  l^{-1}\right)  .
\]

\begin{lemma}
\label{La}The Laplacian-Beltrami operator $\Delta_{\Gamma^{t}}$ on the
hypersurface $\Gamma^{t}$ has the form
\[
\Delta_{\Gamma^{t}}\varphi=\Delta_{\Gamma^{0}}\varphi+P_{1}\left(  l,t\right)
D\varphi+P_{2}\left(  l,t\right)  D^{2}\varphi,
\]
where
\[
\left\vert P_{1}\left(  l,t\right)  \right\vert \leq Cl^{-2},\left\vert
P_{2}\left(  l,t\right)  \right\vert \leq Cl^{-1}.
\]

\end{lemma}

\begin{proof}
Using $\left(  \ref{Laplacian}\right)  ,$ we get
\[
\Delta_{\Gamma^{t}}\varphi=\frac{\partial_{i}\left(  g^{i,j,t}\sqrt{\left\vert
g^{t}\right\vert }\partial_{j}\varphi\right)  }{\sqrt{\left\vert
g^{t}\right\vert }}=\partial_{i}g^{i,j,t}\partial_{j}\varphi+g^{i,j,t}%
\partial_{ij}\varphi+g^{i,j,t}\partial_{j}\varphi\partial_{i}\ln
\sqrt{\left\vert g^{t}\right\vert }.
\]
We compute
\begin{align*}
\Delta_{\Gamma^{t}}\varphi-\Delta_{\Gamma^{0}}\varphi &  =\partial_{i}\left(
g^{i,j,t}-g^{i,j}\right)  \partial_{j}\varphi+\left(  g^{i,j,t}-g^{i,j}%
\right)  \partial_{ij}\varphi\\
&  +g^{i,j,t}\partial_{j}\varphi\partial_{i}\ln\frac{\sqrt{\left\vert
g^{t}\right\vert }}{\sqrt{\left\vert g\right\vert }}+\left(  g^{i,j,t}%
-g^{i,j}\right)  \partial_{j}\varphi\partial_{i}\ln\sqrt{\left\vert
g\right\vert }.
\end{align*}
The estimate follows from this formula.
\end{proof}

One main step of our analysis will be the estimate of the approximate solution.

\begin{lemma}
\label{error}The error of the approximate solution $\bar{H}$ has the following
estimate:
\begin{align*}
\Delta\bar{H}+\bar{H}-\bar{H}^{3}  &  =\left(  h_{d}^{\prime\prime}+\frac
{6}{l}h_{d}^{\prime}\right)  \bar{H}^{\prime}+O\left(  h_{d}^{\prime2}\right)
+O\left(  l^{-2}\right)  h_{d}^{\prime}+O\left(  l^{-1}\right)  h_{d}%
^{\prime\prime}\\
&  +\left(  t\left\vert A\right\vert ^{2}+t^{3}%
{\displaystyle\sum_{i=1}^{7}}
k_{i}^{4}\right)  \bar{H}^{\prime}+O\left(  l^{-5}\right)  .
\end{align*}

\end{lemma}

\begin{proof}
Computing the Laplacian in the Fermi coordinate, we obtain, up to an
exponential decay term introduced by the cutoff function $\rho,$
\begin{align*}
\Delta\bar{H}+\bar{H}-\bar{H}^{3}  &  =\Delta_{\Gamma^{t}}\bar{H}+\partial
_{t}^{2}\bar{H}-M_{t}\partial_{t}\bar{H}+\bar{H}-\bar{H}^{3}\\
&  =\Delta_{\Gamma^{t}}\bar{H}-M_{t}\partial_{t}\bar{H}.
\end{align*}
Let $k_{i}$ be the principle curvatures of $\Gamma_{a}^{+}$. Since $\Gamma
_{a}^{+}$ is a minimal surface,
\begin{align*}
M_{t}  &  =%
{\displaystyle\sum_{i=1}^{7}}
\frac{k_{i}}{1-tk_{i}}\\
&  =t\left\vert A\right\vert ^{2}+t^{2}%
{\displaystyle\sum_{i=1}^{7}}
k_{i}^{3}+t^{3}%
{\displaystyle\sum_{i=1}^{7}}
k_{i}^{4}+O\left(  k_{i}^{5}\right)  .
\end{align*}
Observe that $k_{i}^{3}$ decays like $O\left(  r^{-3}\right)  $ at infinity.
However, we would like to show that $%
{\displaystyle\sum_{i=1}^{7}}
k_{i}^{3}$ actually has a faster decay. Indeed, recall that(see for example
\cite{ALen}), in a parametrization of the curve $\Gamma_{a}^{+},$
\begin{align*}
k_{1}  &  =\frac{-r^{\prime\prime}s^{\prime}+r^{\prime}s^{\prime\prime}%
}{\left(  r^{\prime2}+s^{\prime2}\right)  ^{\frac{3}{2}}},\\
k_{2}  &  =k_{3}=k_{4}=\frac{s^{\prime}}{r\sqrt{r^{\prime2}+s^{\prime2}}},\\
k_{5}  &  =k_{6}=k_{7}=\frac{-r^{\prime}}{s\sqrt{r^{\prime2}+s^{\prime2}}}.
\end{align*}
In particular, using the fact that along $\Gamma_{a}^{+},$ $s=r+ar^{-2}%
+O\left(  r^{-3}\right)  ,$
\[
k_{1}=O\left(  r^{-4}\right)  ,\text{ and }\left\vert k_{2}\right\vert
=...=\left\vert k_{7}\right\vert =\frac{1}{\sqrt{2}r}+O\left(  r^{-4}\right)
.
\]
Therefore we obtain
\[%
{\displaystyle\sum_{i=1}^{7}}
k_{i}^{3}=O\left(  r^{-5}\right)  .
\]
It follows that,%
\[
M_{t}=t\left\vert A\right\vert ^{2}+t^{3}%
{\displaystyle\sum_{i=1}^{7}}
k_{i}^{4}+O\left(  r^{-5}\right)  .
\]

Next we compute $\Delta_{\Gamma^{t}}\bar{H}.$ By Lemma \ref{La}, in the Fermi
coordinate,
\begin{align*}
\Delta_{\Gamma^{t}}\bar{H}  &  =\frac{\partial_{i}\left(  g^{i,j,t}%
\sqrt{\left\vert g^{t}\right\vert }\partial_{j}\left[  H\left(  t-h_{d}\left(
l\right)  \right)  \right]  \right)  }{\sqrt{\left\vert g^{t}\right\vert }}\\
&  =\Delta_{\Gamma^{0}}\bar{H}+O\left(  l^{-2}\right)  h_{d}^{\prime}+O\left(
l^{-1}\right)  h_{d}^{\prime\prime}\\
&  =-\left(  h_{d}^{\prime\prime}+\frac{6}{l}h_{d}^{\prime}\right)  \bar
{H}^{\prime}+O\left(  h_{d}^{\prime2}\right)  +O\left(  l^{-2}\right)
h_{d}^{\prime}+O\left(  l^{-1}\right)  h_{d}^{\prime\prime}+O\left(
l^{-5}\right)  .
\end{align*}
Hence
\begin{align*}
\Delta\bar{H}+\bar{H}-\bar{H}^{3}  &  =-\left(  h_{d}^{\prime\prime}+\frac
{6}{l}h_{d}^{\prime}\right)  \bar{H}^{\prime}+O\left(  h_{d}^{\prime2}\right)
+O\left(  l^{-2}\right)  h_{d}^{\prime}+O\left(  l^{-1}\right)  h_{d}%
^{\prime\prime}\\
&  -\left(  t\left\vert A\right\vert ^{2}+t^{3}%
{\displaystyle\sum_{i=1}^{7}}
k_{i}^{4}\right)  \bar{H}^{\prime}+O\left(  l^{-5}\right)  .
\end{align*}

\end{proof}

Let us set
\[
J\left(  h_{d}\right)  =h_{d}^{\prime\prime}+\frac{6}{l}h_{d}^{\prime}%
+\frac{6}{l^{2}}h_{d}.
\]
With all these understood, we are ready to prove the following

\begin{proposition}
\label{Es3}The function $h_{d}$ satisfies
\[
\left\vert J\left(  h_{d}\right)  \right\vert \leq Cl^{-5},
\]
where $C$ is a constant independent of $d.$
\end{proposition}

\begin{proof}
Frequently, we drop the subscript $d$ for notational simplicity.

Since $\phi+\bar{H}$ solves the Allen-Cahn equation, $\phi$ satisfies
\[
-\Delta\phi+\left(  3\bar{H}^{2}-1\right)  \phi=\Delta\bar{H}+\bar{H}-\bar
{H}^{3}-3\bar{H}\phi^{2}-\phi^{3}.
\]
By Lemma \ref{error},
\begin{align*}
&  -\Delta\phi+\left(  3\bar{H}^{2}-1\right)  \phi\\
&  =-J\left(  h\right)  \bar{H}^{\prime}-\left(  t-h\right)  \bar{H}^{\prime
}\left\vert A\right\vert ^{2}-\left(  t-h\right)  ^{3}\bar{H}^{\prime}%
{\displaystyle\sum_{i=1}^{7}}
k_{i}^{4}\\
&  +O\left(  h^{\prime2}\right)  +O\left(  l^{-2}h^{\prime}\right)  +O\left(
l^{-1}h^{\prime\prime}\right) \\
&  -3\phi^{2}\bar{H}-\phi^{3}+O\left(  l^{-5}\right)  .
\end{align*}
The function $\left(  t-h\right)  \left\vert A\right\vert ^{2}\bar{H}^{\prime
}$ is orthogonal to $\bar{H}^{\prime}$ and decays like $O\left(
l^{-2}\right)  .$ This is a slow decaying term. Recall that we defined a
function $\eta$ satisfying
\[
-\eta^{\prime\prime}+\left(  3H^{2}-1\right)  \eta=-tH^{\prime}\left(
t\right)  .
\]
We introduce $\bar{\eta}=\eta\left(  t-h\right)  .$ Straightforward
computation yields
\begin{align*}
&  -\Delta\left(  \bar{\eta}\left\vert A\right\vert ^{2}\right)  +\left(
3\bar{H}^{2}-1\right)  \bar{\eta}\left\vert A\right\vert ^{2}\\
&  =-\Delta_{\Gamma^{t}}\left(  \bar{\eta}\left\vert A\right\vert ^{2}\right)
-\partial_{t}^{2}\left(  \bar{\eta}\left\vert A\right\vert ^{2}\right)
+M_{t}\partial_{t}\left(  \bar{\eta}\left\vert A\right\vert ^{2}\right)
+\left(  3\bar{H}^{2}-1\right)  \bar{\eta}\left\vert A\right\vert ^{2}\\
&  =-\Delta_{\Gamma^{t}}\left(  \bar{\eta}\left\vert A\right\vert ^{2}\right)
+M_{t}\partial_{t}\left(  \bar{\eta}\left\vert A\right\vert ^{2}\right)
-\left(  t-h\right)  \bar{H}^{\prime}\left\vert A\right\vert ^{2}\\
&  =-\bar{\eta}\Delta_{\Gamma}\left\vert A\right\vert ^{2}+\left(  t-h\right)
\partial_{t}\bar{\eta}\left(  \left\vert A\right\vert ^{2}\right)
^{2}-\left(  t-h\right)  \bar{H}^{\prime}\left\vert A\right\vert ^{2}+O\left(
l^{-5}\right)  .
\end{align*}
Due to the fact that $\eta$ is odd, $-\bar{\eta}\Delta_{\Gamma}\left\vert
A\right\vert ^{2}+\left(  t-h\right)  \partial_{t}\bar{\eta}\left(  \left\vert
A\right\vert ^{2}\right)  ^{2}$ decays like $O\left(  l^{-4}\right)  $ but is
orthogonal to $\bar{H}^{\prime}.$

Let $\phi=\rho\bar{\eta}\left\vert A\right\vert ^{2}+\bar{\phi}.$ Then the new
function $\bar{\phi}$ should satisfy
\begin{align}
-\Delta\bar{\phi}+\left(  3\bar{H}^{2}-1\right)  \bar{\phi}  &  =-J\left(
h\right)  \bar{H}^{\prime}+\left(  t-h\right)  ^{3}H^{\prime}%
{\displaystyle\sum_{i=1}^{7}}
k_{i}^{4}\label{equfi}\\
&  +\bar{\eta}\Delta_{\Gamma}\left\vert A\right\vert ^{2}-\left(  t-h\right)
\partial_{t}\bar{\eta}\left(  \left\vert A\right\vert ^{2}\right)
^{2}\nonumber\\
&  +O\left(  h^{\prime2}\right)  +O\left(  l^{-2}\right)  h^{\prime}+O\left(
l^{-1}\right)  h^{\prime\prime}\nonumber\\
&  +3\left(  \bar{\eta}\left\vert A\right\vert ^{2}+\bar{\phi}\right)
^{2}\bar{H}+\left(  \bar{\eta}\left\vert A\right\vert ^{2}+\bar{\phi}\right)
^{3}+O\left(  l^{-5}\right)  .\nonumber
\end{align}
Denote the right hand by $E.$ Then multiplying both sides with $\bar
{H}^{\prime}$ and integrating in $t$, we get
\[
\int_{\mathbb{R}}E\bar{H}^{\prime}=o\left(  \bar{\phi}\right)  .
\]
>From this estimate, we get
\begin{equation}
-\Delta\bar{\phi}+\left(  3\bar{H}^{2}-1\right)  \bar{\phi}=E-\frac
{\int_{\mathbb{R}}E\bar{H}^{\prime}}{\int_{\mathbb{R}}\bar{H}^{\prime2}}%
\bar{H}^{\prime}+o\left(  \bar{\phi}\right)  . \label{E}%
\end{equation}
Note that
\begin{align*}
E-\frac{\int_{\mathbb{R}}E\bar{H}^{\prime}}{\int_{\mathbb{R}}\bar{H}^{\prime
2}}\bar{H}^{\prime}  &  =\left(  t-h\right)  ^{3}\bar{H}^{\prime}%
{\displaystyle\sum_{i=1}^{7}}
k_{i}^{4}\\
&  +\bar{\eta}\Delta_{\Gamma}\left\vert A\right\vert ^{2}-\left(  t-h\right)
\partial_{t}\bar{\eta}\left(  \left\vert A\right\vert ^{2}\right)  ^{2}\\
&  +O\left(  h^{\prime2}\right)  +O\left(  l^{-2}\right)  h^{\prime}+O\left(
l^{-1}\right)  h^{\prime\prime}\\
&  +3\bar{\eta}^{2}\left(  \left\vert A\right\vert ^{2}\right)  ^{2}\bar
{H}+o\left(  \bar{\phi}\right)  +O\left(  l^{-5}\right)  .
\end{align*}
By the estimate of $h,h^{\prime},h^{\prime\prime},$ we get the following
estimate(non optimal):
\[
E-\frac{\int_{\mathbb{R}}EH^{\prime}}{\int_{\mathbb{R}}H^{\prime2}}H^{\prime
}=O\left(  l^{-3}\right)  +o\left(  \bar{\phi}\right)  .
\]
Since $\bar{\phi}=0$ on $L_{d}$ (This follows from the boundary condition of
$u_{d}.$ Actually $\bar{\phi}$ is not identically zero on $L_{d},$ but of the
order $O\left(  e^{-d}\right)  $), we could use the well known estimate of the
linear theory for the operator
\[
-\Delta+\left(  3H^{2}-1\right)  ,
\]
to deduce that
\begin{equation}
\left\vert \bar{\phi}\right\vert +\left\vert D\phi\right\vert +\left\vert
D^{2}\phi\right\vert \leq Cl^{-3}, \label{fibar}%
\end{equation}
with $C$ independent of $d.$

Now multiplying both sides of $\left(  \ref{equfi}\right)  $ by $\bar
{H}^{\prime}$ again and integrating in $t$, using the estimate $\left(
\ref{fibar}\right)  $ of $\bar{\phi},$ we get
\begin{equation}
J\left(  h\right)  +O\left(  h^{\prime2}\right)  +O\left(  l^{-2}\right)
h^{\prime}+O\left(  l^{-1}h^{\prime\prime}\right)  =\int_{\mathbb{R}}\left[
-\Delta\bar{\phi}+\left(  3\bar{H}^{2}-1\right)  \bar{\phi}\right]  \bar
{H}^{\prime}dt+O\left(  l^{-5}\right)  . \label{Jacobi}%
\end{equation}
We compute,
\begin{align*}
\int_{\mathbb{R}}\left[  -\Delta\bar{\phi}+\left(  3\bar{H}^{2}-1\right)
\bar{\phi}\right]  \bar{H}^{\prime}dt  &  =\int_{\mathbb{R}}\left[
\Delta_{\Gamma^{t}}\bar{\phi}+\partial_{t}^{2}\bar{\phi}-M_{t}\partial_{t}%
\bar{\phi}+\left(  3\bar{H}^{2}-1\right)  \bar{\phi}\right]  \bar{H}^{\prime
}dt\\
&  =\int_{\mathbb{R}}\Delta_{\Gamma^{t}}\bar{\phi}\bar{H}^{\prime}dt+O\left(
l^{-5}\right)  .
\end{align*}
>From this equation, we get%
\begin{equation}
J\left(  h\right)  +O\left(  h^{\prime2}\right)  +O\left(  l^{-2}\right)
h^{\prime}+O\left(  l^{-1}\right)  h^{\prime\prime}=\int_{\mathbb{R}}%
\Delta_{\Gamma^{t}}\bar{\phi}\bar{H}^{\prime}dt+O\left(  l^{-5}\right)  .
\label{h}%
\end{equation}
Using the fact that $\left\vert h^{\prime}\right\vert ,\left\vert
h^{\prime\prime}\right\vert \leq Cl^{-2}$ and
\[
\int_{\mathbb{R}}\Delta_{\Gamma^{t}}\bar{\phi}\bar{H}^{\prime}dt=O\left(
l^{-4}\right)  ,
\]
We deduce from $\left(  \ref{h}\right)  $ that
\begin{equation}
\left\vert h^{\prime}\right\vert \leq Cl^{-3},\left\vert h^{\prime\prime
}\right\vert \leq Cl^{-4}. \label{h''}%
\end{equation}
Insert this estimate back into $\left(  \ref{E}\right)  ,$ we get an improved
estimate for $\bar{\phi}:$%
\[
\left\vert \bar{\phi}\right\vert \leq Cl^{-4}.
\]
It follows that%
\begin{equation}
J\left(  h\right)  =O\left(  l^{-5}\right)  . \label{hnew}%
\end{equation}
This finishes the proof.
\end{proof}

Next we would like to use Proposition \ref{Es3} to analyze the behavior of the
nodal curve of the limiting solution $U.$

\begin{lemma}
\label{b}There exists a constant $b\in\left[  -\lambda^{\ast},\lambda^{\ast
}\right]  $ such that
\[
F\left(  r\right)  =r+br^{-2}+O\left(  r^{-3}\right)  .
\]

\end{lemma}

\begin{proof}
For the limiting solution $U,$ we write it in the Fermi coordinate as
\[
U=\bar{H}\left(  t-h^{\ast}\left(  l\right)  \right)  +\phi^{\ast},
\]
where $\phi^{\ast}$ is orthogonal to $\bar{H}^{\prime}.$ Since we have the
uniform estimate for the function $h_{d}$
\[
J\left(  h_{d}\right)  =O\left(  l^{-5}\right)  ,
\]
We get
\[
J\left(  h^{\ast}\right)  =O\left(  l^{-5}\right)  .
\]
Variation of parameter formula tells us that
\[
h^{\ast}\left(  l\right)  =k_{1}l^{-2}+o\left(  l^{-2}\right)  .
\]
This together with the fact that $U$ satisfies%
\[
U_{\lambda}^{-}\leq U\leq U_{\lambda}^{+}.
\]
completes the proof.
\end{proof}

Next we show that the solution $U$ has the desired asymptotic behavior as
$l\rightarrow+\infty.$

\begin{proof}
[Proof of Proposition \ref{asym}]It suffices to show that $b=a,$ where the
constant $b$ is derived in Lemma \ref{b}.

Let $r_{0}$ be a constant large enough but fixed. Since $u_{d}\rightarrow U,$
\begin{equation}
h_{d}\left(  r_{0}\right)  =\left(  b-a\right)  r_{0}^{-2}+o\left(  r_{0}%
^{-2}\right)  ,\text{ for }d\text{ large}. \label{ab}%
\end{equation}
On the other hand, we have
\[
\left\{
\begin{array}
[c]{l}%
J\left(  h_{d}\right)  =O\left(  l^{-5}\right)  \text{ in }\left(
r_{0},d\right)  ,\\
h_{d}\left(  d\right)  =0.
\end{array}
\right.
\]
Variation of parameters formula tells us that
\[
h_{d}=c_{1}l^{-2}+c_{2}l^{-3}+O\left(  l^{-3}\ln l\right)  ,
\]
where $c_{1}$, $c_{2}$ may depend on $d.$ Taking into the boundary condition
of $h_{d}\left(  d\right)  =0,$ we get%
\begin{equation}
c_{1}+c_{2}d^{-1}=O\left(  d^{-1}\ln d\right)  . \label{c11}%
\end{equation}
On the other hand, $\left(  \ref{ab}\right)  $ tells us that%
\begin{equation}
c_{1}+c_{2}r_{0}^{-1}+O\left(  r_{0}^{-1}\ln r_{0}\right)  +o\left(  1\right)
=b-a. \label{c12}%
\end{equation}
Equation $\left(  \ref{c11}\right)  $ and $\left(  \ref{c12}\right)  $ lead
to
\begin{equation}
\left(  r_{0}^{-1}-d^{-1}\right)  c_{2}=O\left(  1\right)  . \label{c2}%
\end{equation}
Let $d\rightarrow+\infty.$ $\left(  \ref{c2}\right)  $ clearly implies that
$c_{2}$ is bounded, which in turn tells us that $c_{1}\rightarrow0$ as
$d\rightarrow+\infty.$ Hence $b=a.$
\end{proof}

\section{Global minimizers from Lawson's minimizing cones}

We have obtained global minimizers from the minimal surfaces asymptotic to the
Simons cone. In this section, we will perform similar analysis for more
general strictly area minimizing cone. More precisely, we will consider the
Lawson's minimizing cone $C_{i,j}$ mentioned in the first section, where
either
\[
i+j\geq9,
\]
or
\[
i+j=8,\text{ }\left\vert i-j\right\vert \leq4.
\]

Let
\[
r=\sqrt{x_{1}^{2}+...+x_{i}^{2}},\quad s=\sqrt{x_{i+1}^{2}+...+x_{i+j}^{2}%
},\quad l=\sqrt{r^{2}+s^{2}}.
\]
Let $\left\vert A\right\vert ^{2}$ be the squared norm of the second
fundamental form of $C_{i,j}.$ Put $i+j=n.$ The Jacobi operator of $C_{i,j}$,
acting on functions $h\left(  l\right)  $ defined on $C_{i,j}$ which
additionally only depends on $l,$ has the form
\[
J\left(  h\right)  =h^{\prime\prime}+\frac{n-2}{l}h^{\prime}+\frac{n-2}{l^{2}%
}h.
\]
Solutions of the equation $J\left(  h\right)  =0$ is given by
\[
c_{1}l^{\alpha^{+}}+c_{2}l^{\alpha^{-}},
\]
where%
\[
\alpha^{\pm}=\frac{-\left(  n-3\right)  \pm\sqrt{\left(  n-3\right)
^{2}-4\left(  n-2\right)  }}{2}.
\]
One could check that for $n\geq8,$ we always have%
\[
-2\leq\alpha^{+}<-1.
\]
The main result of this section is the following

\begin{proposition}
\label{general}There exists a constant $c_{i,j}$ such that for each
$k\in\mathbb{R},$ we could construct a solution $U_{k}$ to the Allen-Cahn
equation such that for $r$ large, the nodal set of $U_{k}$ has the asymptotic
behavior:
\[
f_{U_{k}}\left(  r\right)  =\sqrt{\frac{j-1}{i-1}}r+c_{i,j}r^{-1}%
+kr^{\alpha^{+}}+o\left(  r^{\alpha^{+}}\right)  .
\]

\end{proposition}

For notational convenience, let us simply consider the cone $C_{3,5}$ over the
product of spheres $S^{2}\left(  \sqrt{\frac{2}{6}}\right)  \times
S^{4}\left(  \sqrt{\frac{4}{6}}\right)  .$ The proof for other cases are
similar. Under a choice of the unit normal, the principle curvature of
$C_{3,5}$ is given by
\[
k_{1}=0, \quad k_{2}=k_{3}=\frac{\sqrt{2}}{l}, \quad k_{4}=k_{5}=k_{6}%
=k_{7}=-\frac{1}{\sqrt{2}l}.
\]
We set $A_{m}:=%
{\displaystyle\sum_{i}}
k_{i}^{m}.$ In particular, $A_{2}:=\left\vert A\right\vert ^{2}=\frac{6}%
{l^{2}}$ and $A_{3}:=%
{\displaystyle\sum}
k_{i}^{3}=\frac{3\sqrt{2}}{l^{3}}.$ It is well known that $C_{3,5}$ is a
strict area minimizing cone. There is also a foliation of $\mathbb{R}^{8}$ by
minimal hypersurfaces asymptotic to $C_{3,5}.$ By slightly abusing the
notation, we still use $\Gamma_{\lambda}^{\pm}$ to denote this foliation. For
$\lambda$ sufficiently large (say $\lambda\geq\lambda_{0}$), the construction
of Pacard-Wei again gives us a family of solutions $U_{\lambda}^{\pm}$ whose
zero level set is close to $\Gamma_{\lambda}^{\pm}.$ The strict area
minimizing assumption on the cone is actually used to ensure that this family
of solutions are ordered.

\begin{lemma}
The family of solutions $U_{\lambda}^{\pm}$ is ordered. That is,
\begin{align*}
U_{\lambda_{1}}^{+}\left(  X\right)   &  <U_{\lambda_{2}}^{+}\left(  X\right)
,\lambda_{1}<\lambda_{2}.\\
U_{\lambda_{1}}^{-}\left(  X\right)   &  <U_{\lambda_{2}}^{-}\left(  X\right)
,\lambda_{1}>\lambda_{2},\\
U_{\lambda_{0}}^{-}\left(  X\right)   &  <U_{\lambda_{0}}^{+}\left(  X\right)
.
\end{align*}

\end{lemma}

\begin{proof}
Since this has not been proven in the paper \cite{PW}, we give a sketch of the proof.

We only consider the family of solutions $U_{\lambda}^{+},$ which we simply
write it as $U_{\lambda}.$ $U_{\lambda}$ is obtained from Lyapunov-Schmidt
reduction. Adopting the notations of the previous sections, let $\left(
l,t\right)  $ be the Fermi coordinate with respect to $\Gamma_{\lambda}%
=\Gamma_{\lambda}^{+}.$ Then in the Fermi coordinate, $U_{\lambda}$ has the
form%
\[
U_{\lambda}=H\left(  t-h\left(  l\right)  \right)  +\phi:=\bar{H}+\phi.
\]
where $\phi$ is orthogonal to $\bar{H}^{\prime}.$

Similarly as before, we know that $\phi$ satisfies%
\[
-\Delta\phi+\left(  3\bar{H}^{2}-1\right)  \phi=\Delta_{\Gamma^{t}}\bar
{H}-M_{t}\bar{H}^{\prime}+o\left(  \phi\right)  .
\]
For notational convenience, we set $\bar{t}=t-h\left(  l\right)  .$ Let us
assume for this moment that $\left\vert h\right\vert \leq Cl^{-1}.$ Recall
that
\begin{align*}
M_{t}  &  =\left(  tA_{2}-t^{2}A_{3}+t^{3}A_{4}\right)  \bar{H}^{\prime
}+O\left(  l^{-5}\right) \\
&  =\left(  \bar{t}A_{2}-\bar{t}^{2}A_{3}+\bar{t}^{3}A_{4}\right)  \bar
{H}^{\prime}+h\bar{H}^{\prime}-2\bar{t}\bar{H}^{\prime}hA_{3}+O\left(
l^{-5}\right)  .
\end{align*}
Inspecting the projection of these terms onto $\bar{H}^{\prime},$ we find that
the main order term of the projection should be
\[
-\bar{t}^{2}\bar{H}^{\prime}A_{3}+h\bar{H}^{\prime}.
\]
Next let us compute the term $\Delta_{\Gamma^{t}}\bar{H}.$ First of all,
\begin{align*}
\Delta_{\Gamma^{0}}\bar{H}  &  =\frac{\partial_{i}\left(  g^{i,j}%
\sqrt{\left\vert g\right\vert }\partial_{j}\left[  H\left(  t-h\left(
l\right)  \right)  \right]  \right)  }{\sqrt{\left\vert g\right\vert }}\\
&  =-\left(  h^{\prime\prime}+\frac{6}{l}h^{\prime}\right)  \bar{H}^{\prime
}+g^{11}H^{\prime\prime}\left(  \bar{t}\right)  h^{\prime2}+O\left(
l^{-5}\right)  .
\end{align*}
Hence%
\begin{align*}
\Delta_{\Gamma^{t}}\bar{H}  &  =-\left(  h^{\prime\prime}+\frac{6}{l}%
h^{\prime}\right)  \bar{H}^{\prime}+H^{\prime\prime}\left(  \bar{t}\right)
O\left(  h^{\prime2}\right)  +O\left(  l^{-5}\right) \\
&  +\bar{t}\bar{H}^{\prime}O\left(  h^{\prime\prime}l^{-1}+h^{\prime}%
l^{-2}\right)  .
\end{align*}
As a consequence, the function $\phi$ should satisfy%
\begin{align*}
-\Delta\phi+\left(  3\bar{H}^{2}-1\right)  \phi &  =\Delta_{\Gamma_{t}}\bar
{H}-M_{t}\bar{H}^{\prime}+o\left(  \phi\right) \\
&  =-J\left(  h\right)  \bar{H}^{\prime}+H^{\prime\prime}\left(  \bar
{t}\right)  O\left(  h^{\prime2}\right)  +\bar{t}\bar{H}^{\prime}O\left(
h^{\prime\prime}l^{-1}+h^{\prime}l^{-2}\right) \\
-  &  \left(  \bar{t}A_{2}-\bar{t}^{2}A_{3}+\bar{t}^{3}A_{4}\right)  \bar
{H}^{\prime}+2\bar{t}\bar{H}^{\prime}hA_{3}+O\left(  l^{-5}\right)  +o\left(
\phi\right)  .
\end{align*}
Projecting onto $\bar{H}^{\prime}$, the main order term at the right hand side
is
\[
-J\left(  h\right)  \bar{H}^{\prime}+\bar{t}^{2}A_{3}\bar{H}^{\prime}.
\]
Hence we find that the main order term $h_{0}$ of $h$ should satisfy the
equation
\[
J\left(  h_{0}\right)  =c^{\ast}A_{3},
\]
where
\[
c^{\ast}=\frac{\int_{\mathbb{R}}t^{2}H^{\prime2}dt}{\int_{\mathbb{R}}%
H^{\prime2}dt}>0.
\]
Let $\bar{h}_{0}\left(  l\right)  =h_{0}\left(  \lambda l\right)  .$ We find
that $\bar{h}_{0}$ should satisfy
\[
\bar{J}\left(  \bar{h}_{0}\right)  =\frac{1}{\lambda}c^{\ast}A_{3}.
\]
Here $\bar{J},\bar{k}_{i}$ are the Jacobi operator and principle curvatures
corresponding to the rescaled minimal surface $\Gamma_{1}^{+}.$ Using the
invertibility of the Jacobi operator $\bar{J},$ we could assume the existence
of function $\xi$ solving
\[
\bar{J}\left(  \xi\right)  =c^{\ast}A_{3},
\]
with the asymptotic behavior%
\begin{equation}
c_{0}l^{-1}+o\left(  l^{-2}\right)  . \label{cbar}%
\end{equation}
In this way, we deduce that main order of $U_{\lambda}$ is $H\left(
t-\frac{1}{\lambda}\xi\left(  \frac{l}{\lambda}\right)  \right)  .$

Fix a $\lambda$ large. For each $\delta$ small, there is a solution
$U_{\lambda\left(  1+\delta\right)  }$ of the Allen-Cahn equation associated
to the minimal hypersurface $\lambda\left(  1+\delta\right)  \Gamma_{1}.$ We
will denote it by $u_{\delta}.$ To prove the order property of the family of
solutions, it will be suffice for us to show that for $0<\delta_{1}<\delta
_{2}$ sufficiently small,
\begin{equation}
u_{\delta_{1}}\left(  x\right)  <u_{\delta_{2}}\left(  x\right)  ,\text{ }%
x\in\mathbb{R}^{8}. \label{order}%
\end{equation}

Let us use $t_{\delta}$ to denote the signed distance of a point to
$\Gamma_{\lambda\left(  1+\delta\right)  }.$ The previous analysis tells us
that the main order of $u_{\delta}$ is
\[
H\left(  t_{\delta}-\frac{1}{\lambda\left(  1+\delta\right)  }\xi\left(
\frac{l_{\delta}}{\lambda\left(  1+\delta\right)  }\right)  \right)  .
\]
Take a large constant $k.$ Let $\Xi_{k}$ be a radius $k$ tubular neighbourhood
of $\Gamma_{\lambda}.$ We claim that for each point $P\in\Xi_{k},$
\begin{equation}
t_{\delta_{1}}-\frac{1}{\lambda\left(  1+\delta_{1}\right)  }\xi\left(
\frac{l_{\delta_{1}}}{\lambda\left(  1+\delta_{1}\right)  }\right)
<t_{\delta_{1}}-\frac{1}{\lambda\left(  1+\delta_{2}\right)  }\xi\left(
\frac{l_{\delta_{2}}}{\lambda\left(  1+\delta_{2}\right)  }\right)  .
\label{dist}%
\end{equation}
Indeed, taking into account of the fact that
\[
\xi\left(  l\right)  =\frac{1}{1+l}+o\left(  \left(  1+l\right)  ^{-2}\right)
,
\]
we get, for $\varepsilon$ small,
\begin{equation}
\varepsilon\xi\left(  \varepsilon l\right)  =\frac{\varepsilon}{1+\varepsilon
l}+\varepsilon O\left(  \left(  1+\varepsilon l\right)  ^{-2}\right)  .
\label{d1}%
\end{equation}
On the other hand, for $\delta_{1},\delta_{2}$ sufficiently small(depending on
$\lambda$),
\begin{equation}
t_{\delta_{1}}-t_{\delta_{2}}\geq c\lambda^{3}\left(  \delta_{2}-\delta
_{1}\right)  \frac{1}{1+l^{2}}, \label{d2}%
\end{equation}
for some constant $c.$ The inequality $\left(  \ref{dist}\right)  $ then
follows from $\left(  \ref{d1}\right)  $ and $\left(  \ref{d2}\right)  .$ Once
we have $\left(  \ref{dist}\right)  ,$ the same argument as in the last
section of \cite{PW} applies and $\left(  \ref{order}\right)  $ is proved.
\end{proof}

By this lemma, the family of solutions $U_{\lambda}^{\pm}$ forms a foliation.
We could use them as sub and super solutions to obtain solutions between them
and we have similar results as in the Simons cone case. However, in the
current situation, we show that the nodal set of each solution will be
asymptotic to the curve $s=\sqrt{2}r+\frac{c_{0}}{\sqrt{3}}r^{-1},$ where
$c_{0}$ is the constant appearing in $\left(  \ref{cbar}\right)  .$

Now we are ready to prove Proposition \ref{general}. We still focus on the
case $\left(  i,j\right)  =\left(  3,5\right)  .$\ Since the main steps are
same as the case of Simons' cone, we shall only sketch the proof and point out
the main difference.

\begin{proof}
[Proof of Proposition \ref{general}]Let $k\in\mathbb{R}$ be a fixed real
number. Let $\left(  l,t\right)  $ be the Fermi coordinate with respect to the
minimal hypersurface asymptotic to the cone $C_{3,5}$ with the asymptotic
behavior
\[
s=f_{k}\left(  r\right)  :=\sqrt{2}r+kr^{-2}+o\left(  r^{-2}\right)  .
\]

We could construct minimizers on a sequences of bounded domain $\Omega_{d}.$
Let $L_{d}$ be the line orthogonal to the minimal surface at $\left(
d,f_{k}\left(  d\right)  \right)  .$ On $L_{d}$ we impose suitable Dirichlet
boundary condition that
\[
u_{d}=H\left(  t-\xi\left(  l\right)  \right)  +\eta\left(  t\right)
\left\vert A\right\vert ^{2},
\]
at least away from the axes. Recall that we have ordered solutions of
Pacard-Wei. We could assume that the boundary function is trapped between two
solutions $U_{\lambda\ast}^{+}$ and $U_{\lambda^{\ast}}^{-}$ for some fixed
$\lambda^{\ast}.$ Then the minimizers $u_{d}$ on $\Omega_{d}$ will be between
$U_{\lambda^{\ast}}^{+}$ and $U_{\lambda^{\ast}}^{-}.$ Let us now take the
limit for the sequence of solutions $\left\{  u_{d}\right\}  $ obtained in
this way, as $d\rightarrow+\infty.$

We need to analyze the asymptotic behavior of $\left\{  u_{d}\right\}  .$
Define the approximate solution $\bar{H}\left(  t-h\right)  $ as before and
write $u_{d}=\bar{H}+\phi.$ Then we get
\begin{align*}
-\Delta\phi+\left(  3\bar{H}^{2}-1\right)  \phi &  =\Delta_{\Gamma_{t}}\bar
{H}-M_{t}\bar{H}^{\prime}+o\left(  \phi\right) \\
&  =-J\left(  h\right)  \bar{H}^{\prime}+H^{\prime\prime}\left(  \bar
{t}\right)  O\left(  h^{\prime2}\right)  +\bar{t}\bar{H}^{\prime}O\left(
h^{\prime\prime}l^{-1}+h^{\prime}l^{-2}\right) \\
-  &  \left(  \bar{t}A_{2}-\bar{t}^{2}A_{3}+\bar{t}^{3}A_{4}\right)  \bar
{H}^{\prime}+2\bar{t}\bar{H}^{\prime}hA_{3}+O\left(  l^{-5}\right)  +o\left(
\phi\right)  .
\end{align*}
Write $h\left(  l\right)  $ as $h^{\ast}\left(  l\right)  +\xi\left(
l\right)  .$ We would like to show that $h^{\ast}=kl^{-2}+O\left(
l^{-3}\right)  .$ Indeed,
\begin{align*}
E  &  :=\Delta\bar{H}+\bar{H}-\bar{H}^{3}=-J\left(  h\right)  \bar{H}^{\prime
}+H^{\prime\prime}\left(  \bar{t}\right)  O\left(  h^{\prime2}\right)
+\bar{t}\bar{H}^{\prime}O\left(  h^{\prime\prime}l^{-1}+h^{\prime}%
l^{-2}\right) \\
-  &  \left(  \bar{t}A_{2}-\bar{t}^{2}A_{3}+\bar{t}^{3}A_{4}\right)  \bar
{H}^{\prime}+2\bar{t}\bar{H}^{\prime}hA_{3}+O\left(  l^{-5}\right)  .
\end{align*}
Inspecting each term in this error, it turns out that the projection of
$E\left(  \bar{H}\right)  $ onto $\bar{H}^{\prime}$ is
\[
J\left(  h^{\ast}\right)  \bar{H}^{\prime}+O\left(  l^{-5}\right)  .
\]
On the other hand, since $u=\bar{H}+\phi$ and $\phi$ satisfies%
\[
-\Delta\phi+\left(  3\bar{H}^{2}-1\right)  \phi=E-\frac{\int_{\mathbb{R}}%
E\bar{H}^{\prime}}{\int_{\mathbb{R}}\bar{H}^{\prime2}}\bar{H}^{\prime
}+o\left(  \phi\right)  ,
\]
we find that
\[
\phi=\eta\left(  \bar{t}\right)  \left\vert A\right\vert ^{2}+\phi^{\ast},
\]
with
\[
\left\vert \phi^{\ast}\right\vert \leq Cl^{-3}.
\]
A refined analysis shows that the $O\left(  r^{-3}\right)  $ term could be
written as $\eta_{2}\left(  \bar{t}\right)  A_{3},$ where $\eta_{2}$
satisfies
\[
-\eta_{2}^{\prime\prime}\left(  t\right)  +\left(  3H^{2}-1\right)  \eta
_{2}=t^{2}H^{\prime}-\frac{\int_{\mathbb{R}}t^{2}H^{\prime2}dt}{\int%
_{\mathbb{R}}H^{\prime2}dt}H^{\prime}.
\]
It then follows from similar arguments as in the previous sections that
\[
J\left(  h^{\ast}\right)  =O\left(  l^{-5}\right)  .
\]
This provides us sufficient estimate to prove our result, proceeding similarly
as in Section \ref{2.3}.
\end{proof}

\end{document}